\documentclass[11pt, reqno]{amsart}%
\usepackage{amsmath,amsfonts,amssymb,amsthm,color,cancel}
\usepackage{epsfig}
\usepackage[normalem]{ulem}

\usepackage{bbm}
\usepackage{enumerate}
\usepackage{amstext}
\usepackage{mathrsfs}
\usepackage{xspace}
\usepackage{amsfonts}
\usepackage{amsmath}
\usepackage{amssymb}
\usepackage{amstext}
\usepackage{amsthm}       
\usepackage{xspace}
\usepackage[active]{srcltx}
\usepackage{cite}

\newcommand{\dprod}[2]{\langle #1,#2\rangle}

\newcommand{\D}[1]{\mbox{\rm #1}}
\newcommand{\cal}{\mathcal}

\newcommand{\F}{\mathcal F}

\newcommand{\R}{\mathbb R}

\newcommand{\Lip}{\D{Lip}}

\newcommand{\N}{\mathbb N}

\newcommand{\PP}{\mathbb{P}}

\newcommand{\Q}{\mathbb Q}

\newcommand{\Tr}{\D{tr}}
\newcommand{\Z}{\mathbb Z}

\newcommand{\comp}{\mbox{\scriptsize  $\circ$}}
\newcommand{\eps}{\varepsilon}
\renewcommand{\epsilon}{\varepsilon}

\newcommand{\cyl}{(0,+\infty)\times\R^d}
\newcommand{\ccyl}{[0,+\infty)\times\R^d}

\newcommand{\cTcyl}{[0,T]\times\R^d}

\newcommand{\Bor}{\mathscr{B}}

\newcommand{\K}{\mathcal{K}}

\newcommand{\tagliato}{$\kern-5 mm -$}
\newcommand{\tagliat}{$\kern-4 mm -$}

\newcommand{\dd}{\D{d}}

\newcommand{\ucv}{\rightrightarrows_{\text{\tiny loc}}}

\newcommand{\Ham}{\mathcal{H}(\gamma,\alpha_0,\beta_0)}
\newcommand{\Hamtilde}{\mathcal{H}(\gamma,\tilde\alpha_0,\tilde\beta_0)}

\newtheorem{teorema}{Theorem}[section]
\newtheorem{prop}[teorema]{Proposition}
\newtheorem{lemma}[teorema]{Lemma}
\newtheorem{definition}[teorema]{Definition}
\newtheorem{cor}[teorema]{Corollary}

\newtheorem{guess}[teorema]{Remark}
\newtheorem{example}[teorema]{Example}

\newenvironment{oss}{\begin{guess} \begin{rm}}{\end{rm} \end{guess}}
\newenvironment{definizione}{\begin{definition} \begin{rm}}{\end{rm}
\end{definition}}
\begin{document}

\title{Homogenization of viscous and non-viscous\\ HJ equations: a remark and an application}
\author{Andrea Davini \and Elena Kosygina}
\address{Dip. di Matematica, {Sapienza} Universit\`a di Roma,
P.le Aldo Moro 2, 00185 Roma, Italy}

\email{davini@mat.uniroma1.it}
\address{Department of Mathematics, Baruch College, One Bernard Baruch Way, Box B-630, New York, NY 10010, USA}

\email{elena.kosygina@baruch.cuny.edu}
\keywords{Homogenization, equations in media with random structure, nonconvex Hamilton-Jacobi equation.}
\subjclass[2010]{35B27, 60K37, 35D40.}

\begin{abstract}
  It was pointed out {by P.-L.\ Lions, G.\ Papanicolaou, and S.R.S.\
    Varadhan in their seminal paper} \cite{LPV} that, for first order
  Hamilton-Jacobi (HJ) equations, homogenization starting with affine
  initial data implies homogenization for general uniformly continuous
  initial data. The argument makes use of some properties of the HJ
  semi-group, in particular, the finite speed of propagation. The last
  property is lost for viscous HJ equations.  In this paper we prove
  the above mentioned implication in both viscous and non-viscous
  cases.  Our proof relies on a variant of Evans's perturbed test
  function method.  As an application, we show homogenization in the
  stationary ergodic setting for viscous and non-viscous HJ equations
  in one space dimension with non-convex Hamiltonians of specific
  form.  The results are new in the viscous case.
 \end{abstract}
\date{\today}
\maketitle
%

\section{Introduction}
Consider a family of equations of the form 
\begin{equation}\label{intro viscous hj}\tag{HJ$_\eps$}
{\partial_t u^\eps}-\eps\,\Tr\left(A\left(\frac{x}{\eps}\right)D_x^2u^\eps\right)+H\left(\frac x \eps, D_x u^\eps\right)=0\quad \hbox{in $ (0,+\infty)\times \R^d$,}
\end{equation}
where { $\epsilon>0$,} $A$ is a $d\times d$ symmetric and
  positive semi-definite matrix with Lipschitz and bounded
  coefficients, and $H$, the Hamiltonian, is a continuous function on
  $\R^{d}\times\R^d$, coercive in the gradient variable, uniformly
  with respect to $x$. If $A\not\equiv 0$, we shall refer to
  \eqref{intro viscous hj} as viscous Hamilton-Jacobi equation.

Under suitable assumptions on $H$, equation \eqref{intro
  viscous hj} satisfies a comparison principle, yielding the existence
of a unique viscosity solution $u^\eps$ (in a proper class of
continuous functions) subject to the initial condition $u^\epsilon(0,\cdot)=g$ in $\R^d$, with $g$ uniformly 
continuous in $\R^d$. 
We shall say that {\em the equation \eqref{intro viscous hj} homogenizes} if
there exists a continuous function $\overline H:\R^d\to\R$ such that
$u^\eps$ converges, locally uniformly in $\ccyl$ as $\eps\to 0^+$, to
the unique viscosity solution $\overline u$ of the following {\em
  effective equation}
\begin{equation}\label{intro eq effective}
\partial_t\overline u+\overline H(D_x\overline u)=0\qquad\hbox{in $\cyl$}
\end{equation}
satisfying $\overline u(0,\cdot)=g$ for every uniformly continuous function $g$ on $\R^d$. 

The study of homogenization of Hamilton-Jacobi equations was
initiated by P.-L.\ Lions, G.\ Papanicolaou, and S.\,R.\,S.\ Varadhan
around  1987. Their seminal paper \cite{LPV} was concerned with
homogenization of  first order Hamilton-Jacobi equations in the
periodic setting, i.\ e.\ when $A\equiv 0$ and $H(\cdot+z,\cdot)\equiv H(\cdot,\cdot)$ for all
$z\in\Z^d$. In particular, Section I.2 of \cite{LPV} explains why, 
for first order Hamilton-Jacobi equations, 
homogenization for linear initial data implies homogenization for general uniformly continuous initial
data. The outline of the proof provided in \cite{LPV} uses characterization results
for strongly continuous semi-groups on $\D{UC}(\ccyl)$, see
\cite{L84,LN}, as well as a uniform (in $\eps$) speed of propagation
for the semigroup generated by the Cauchy problem associated to
\eqref{intro viscous hj}, which holds true since $A\equiv 0$.

In this paper, we give a proof of this fact for both viscous and
non-viscous Hamilton-Jacobi equations under a quite general set of
assumptions on the equation \eqref{intro viscous hj}.  More precisely,
we will prove the following result (Theorem \ref{app teo
    hom}): assume that
\begin{equation}\label{intro linear data}
u^\eps_\theta(t,x)\underset{\eps\to 0^+}{\to} \langle \theta,x\rangle-t\overline
H(\theta)\qquad\hbox{locally uniformly in $[0,+\infty)\times \R^d$} 
\end{equation}
for every $\theta\in\R^d$ and for some continuous and coercive
function $\overline H:\R^d\to\R$, where $u^\eps_\theta$ is the
solution to \eqref{intro viscous hj} with initial datum
$u^\eps_\theta(0,x)=\langle \theta,x\rangle$.\footnote{When $\eps=1$,
  we shall simply write $u_\theta$ in place of $u^1_\theta$.}  Then
equation \eqref{intro viscous hj} homogenizes. Note that, if
homogenization takes place, then the effective Hamiltonian is
completely characterized in terms of the limit of
$u^\epsilon_\theta(1,0)=\eps u_\theta(1/\eps,0)$ as $\eps\to 0^+$,
with equality holding due to the identity
$u^\eps_\theta(t,x)=\eps\, u_\theta(t/\eps,x/\eps)$ on $\ccyl$.  Our
proof relies on a variant of the elegant and powerful perturbed test
function method due to L.\,C.\ Evans \cite{Evans89}, where the test
function is perturbed { by} a term of the form
$\epsilon v_\theta(t/\epsilon,x/\epsilon)$ with
$v_\theta(t,x):=u_\theta(t,x)-\langle
\theta,x\rangle+t\overline{H}(\theta)$
for a proper $\theta\in\R^d$. We recall that the standard
{ homogenization} approach consists in choosing as
$v_\theta$ a (time-independent) sub-linear solution of the cell
problem
\[
-\Tr\left(A\left({x}\right)D_x^2 v\right)+
H(x,\theta+D_x v)=\overline H(\theta)\quad\hbox{in $\R^d$},
\]
also known as {\em (exact) corrector}. In the periodic setting,
{  correctors}  always exist
and are, { moreover}, periodic, but in more general settings
sub-linear correctors need not exist as shown in
\cite{LS_correctors}.
Thus, loosely speaking, 
$\{v_\theta(t,x)\,:\,\theta\in\R^d\}$ can be thought of as a family of $t$-dependent correctors. 

Besides the beauty and simplicity of this abstract result {\em in se},
our main interest is motivated by applications to the stationary
ergodic setting. In Section \ref{sez stationary ergodic framework} we
{take}  a further step  and show that, in order to have
\eqref{intro linear data} with probability one, it suffices to check
that \
$\lim_{\eps\to 0^+} u^\eps_\theta(1,0,\omega)=-\overline H(\theta)$ \
almost surely with respect to $\omega$. 
Our general results are applied in Section  \ref{sez non convex homogenization} to obtain
homogenization for a one dimensional Hamilton-Jacobi equation of the form 
\begin{equation}\label{intro eq parabolic}\tag{HJ$_\epsilon^\omega$}
u^\eps_t-\eps A\left(\frac x\eps\right) u^\eps_{x\,x}+H\left(\frac{x}{\eps}, u^\eps_x,\omega\right)=0\quad\hbox{in $(0,+\infty)\times\R$},
\end{equation}
where the stationary random field $H:\Omega\to \D{C}(\R^d\times\R^d)$
takes values in a special class of non-convex and uniformly
superlinear Hamiltonians. More precisely, we will assume that $H$ is
{\em pinned} at finitely many values $p_1<p_2<\dots<p_n$, meaning that
$H(\cdot,p_i,\cdot)$ is constant on $\R\times\Omega$ for each fixed
$i$ (see Definition \ref{def pinned}), and {\em piecewise convex} in
$p$, meaning that $H(x,\cdot,\omega)$ is convex on each of the
intervals $(-\infty,p_1),\ (p_1,p_2),\dots,\ (p_{n-1},+\infty)$, for
every fixed $(x,\omega)$, see Theorem \ref{teo hom final}.  When
$A\equiv 0$, we can weaken the convexity assumption to
  level-set convexity.

In order to obtain this result, we consider first the case of a stationary Hamiltonian {which is  
pinned at $p=0$.} Such a Hamiltonian can be always written as 
\begin{equation}
  \label{1pin}
  H(x,p,\omega)=\min\{H_-(x,p,\omega),H_+(x,p,\omega)\}=
\begin{cases}
  H_-(x,p,\omega)&\text{if } p\leqslant 0
  \\H_+(x,p,\omega)&\text{if } p\geqslant 0,
\end{cases}
\end{equation} 
where $H_\pm:\Omega\to \D{C}(\R^d\times\R^d)$ are stationary random fields, uniformly coercive in $p$, and satisfying\  
$H_\pm(\cdot,0,\cdot)\equiv h_0$\  on\ \ $\R\times\Omega$\ \ for some constant $h_0\in\R$. The core of our argument consists in showing that, for every fixed $\omega$, the function
$u^\eps_\theta$ enjoys the same kind of monotonicity as its initial
datum $\theta x$ with respect to the $x$ variable. In particular,
$u^\eps_\theta$ is also a solution of
\begin{equation}\label{intro eq convex parabolic}
u^\eps_t-\eps A\left(\frac x\eps\right) u^\eps_{x\,x}+H_\pm\left(\frac{x}{\eps}, u^\eps_x,\omega\right)=0\quad\hbox{in $(0,+\infty)\times\R$},
\end{equation}
according to the sign of $\theta$. If we assume, in addition, that 
the equation \eqref{intro eq convex parabolic} homogenizes for both $H_+$ 
and $H_-$, then we immediately conclude that \eqref{intro eq parabolic} homogenizes for all linear initial data $g(x)=\theta x$. An application
of our previous results gives homogenization of \eqref{intro eq parabolic} for general uniformly continuous  
initial data. Furthermore, if we denote by $\overline{H}_{\pm}$ the
corresponding effective Hamiltonians associated to $H_\pm$, respectively, then the effective Hamiltonian $\overline H$ can be
expressed by the following formula
\[
\overline{H}(\theta)=\min\{\overline{H}_-(\theta),\overline{H}_+(\theta)\}\qquad\hbox{for every $\theta\in\R$}.
\]
The assumption that \eqref{intro eq convex parabolic} homogenizes is,
for instance, fulfilled when $H_\pm$ are convex, or even level-set
convex when $A\equiv 0$, in view of known homogenization results
\cite{AT15, DS09, KRV, LS_viscous,RT,Sou99}. In this case, we conclude 
that equation \eqref{intro eq parabolic} homogenizes when $H$ is of
the form \eqref{1pin}. 
In particular, we infer that $\overline H$ can be neither
convex nor even level-set convex, see Remark \ref{oss non convex
  effective H}.

The case when $H$ is pinned at $p_0\not=0$ can be always
reduced to the one considered above by replacing $H$
with $H(\cdot,p_0+\cdot,\cdot)$, see Remark \ref{shifts}. The
extension of the homogenization result to piecewise convex stationary
Hamiltonians with multiple pinning points is obtained by
  induction on the number of pinning points, see Theorem \ref{teo hom
  final}.  The basic idea is that a piecewise convex stationary
Hamiltonian with $n$ pinning points can be always written in the form
\eqref{1pin} for some $H_\pm$ of same type but with fewer pinning points.

Although we are able to treat only a special family of Hamiltonians in
one dimension, the results are new in the viscous case.  We stress
that the Hamiltonians we consider are typically {neither
  level-set convex nor satisfy any homogeneity condition with respect
  to $p$,} and are, thus, not covered by  examples
{treated} in \cite{Feh} or
\cite[equation (1.6)]{AC}. 
{Even though} the results of \cite{AC} hold in all
dimensions, they  require a finite range
dependence condition on the coefficients. The last assumption is
typically considered to be very restrictive but, as it was recently
demonstrated in \cite{Zil}, homogenization in general stationary
ergodic settings and dimensions larger than one need not hold for
non-convex but otherwise ``standard'' Hamiltonians without some
condition on the decay of correlations of the coefficients. Earlier
works on non-convex homogenization in the stationary ergodic setting
include homogenization for level-set convex Hamiltonians in the
non-viscous case in one space dimension \cite{DS09} and in any
dimension \cite{AS}, see also \cite{Feh} for some additional results
and extensions to viscous case. The first example of homogenization
for a class of Hamiltonians which are not level-set convex was given
in \cite{ATY_nonconvex} for the non-viscous case in all dimensions.
Papers \cite{ATY_1d} and \cite{Gao} provide quite general non-convex
homogenization results for one-dimensional non-viscous HJ
equations.\medskip

We end this introduction by comparing our condition \eqref{intro
  linear data} with a notion of ergodicity introduced in \cite{AB03}
in the context of periodic homogenization. Let us set
$F(x,p,X):=-\Tr(A(x)X)+H(x,p)$ and assume that $F$ is $\Z^d$-periodic
in $x$. Following \cite{AB03}, the function $F$ is said to be {\em
  ergodic at $\theta\in\R^d$} if the periodic solution $w_\theta$ of
\begin{equation*}
  w_t-\Tr \left(A(y)D_x^2w\right)+H(x,\theta+D_x w)=0 \quad\text{in }(0,+\infty)\times\R^d 
\end{equation*}
with initial condition $w(0,\cdot)=0$\ on $\R^d$\ satisfies \
$w_\theta(t,x)/t\to c$\ as $t\to+\infty$ uniformly in $x$, where
$c=c(\theta)$ is a constant.  It was shown in \cite[Section 2.5]{AB03}
that the ergodicity of $F$ at each $\theta\in\R^d$ implies that
(HJ$_\epsilon$) homogenizes, with $\overline H(\theta):=-c(\theta)$
for every $\theta\in\R^d$.  This holds, of course, under proper
assumptions on the parabolic equation associated with $F$, that are
for instance fulfilled when $A$ and $H$ satisfy our standing
assumptions (A1)-(A2) and (H1)-(H2), respectively, and the Cauchy
problem associated to \eqref{intro viscous hj} is well-posed in a
suitable class of continuous functions (see Section \ref{sez notation}
for more details).
To see a connection with our results, observe that the above notion of
ergodicity can be thought of as a version of homogenization of
(HJ$_\epsilon$) for linear initial data. Indeed, note that 
$u^\epsilon_\theta(t,x)=\langle\theta,x\rangle+ \epsilon w_\theta(t/\epsilon,
x/\epsilon)$.  
The ergodicity is equivalent to the
statement that, for every fixed $t>0$,
\begin{equation}
  \label{lidc}
  \lim_{\epsilon \to 0+}u^\epsilon_\theta(t,x)=\langle
\theta,x\rangle-t\,\overline{H}(\theta)\quad \text{uniformly in }x\in\R^d.
\end{equation}
Thus, the above cited result from \cite{AB03} can be restated as
follows: if the convergence \eqref{lidc} takes place for every fixed
$t>0$ and $\theta\in\R^d$, then (HJ$_\epsilon$) homogenizes.  Taking
into account that we do not assume { that} $A$ and $H$
{  are} periodic in $x$ and, thus, forgo the advantages of
having $x$ in a compact set, it is clear that our Theorem~\ref{app teo
  hom} is a close cousin of the quoted result from
\cite{AB03}.\medskip

The paper is organized as follows. Section~\ref{sez notation} is
devoted to preliminaries: basic notation, definitions, comparison
principles, existence and properties of solutions to \eqref{intro
  viscous hj}, both in the viscous and in the non-viscous
case. Section~\ref{sez general} contains our first results concerning
the connection between homogenization and homogenization with linear
initial data.  In Section~\ref{sez stationary ergodic homogenization}
we introduce the stationary ergodic formulation, present a stationary
ergodic version of Theorem~\ref{app teo hom}, and use 
blue our general results to show homogenization for a class of
non-convex viscous HJ equations in one space dimension.\medskip

\smallskip\indent{\textsc{Acknowledgments.}} This collaboration
originates from discussions the authors had during the Research
Program {\em Homogenization and Random Phenomenon} at the
Institut Mittag-Leffler (1 September- 12 December, 2014).  
The authors wish to thank the
organizers for the invitation and
the Institut Mittag-Leffler for hospitality,
stimulating research atmosphere, and financial
support. {The second author would also like to thank
  Yifeng Yu for discussions.}  The authors were partially supported by
the Simons Foundation through a Collaboration Grant for Mathematicians
\#209493. The second author was also partially supported by INdAM -
GNAMPA Research Project 2016 {\em Fenomeni asintotici e
  omogeneizzazione} and by Universit\`a di Roma La Sapienza - Research
Funds 2013.

\numberwithin{equation}{section}

\section{Preliminaries}\label{sez notation}

Throughout the paper, we denote by $\langle\cdot,\cdot\rangle$
and $|\cdot|$, respectively, the scalar product and the Euclidean norm
on $\R^d$, $d\in\N$.  We let $B_R(x_0)$ and $B_R$ be the open balls in
$\R^d$ of radius $R$ centered at $x_0$ and $0$, respectively. For a
given subset $E$ of $\R^{d}$, we will denote by $\overline E$ its
closure. 

By modulus of continuity we mean a nondecreasing function from $[0,+\infty)$ to
$[0,+\infty)$, vanishing and continuous at $0$. 

Given a metric space $X$, we write $\varphi_n\ucv\varphi$ on $X$
{ when}  the sequence of functions $(\varphi_n)_n$ uniformly
converges to $\varphi$ on compact subsets of $X$.  We  denote by
 $\Lip(X)$, $\D{UC}(X)$, $\D{LSC}(X)$, and $\D{USC}(X)$ the space of Lipschitz continuous, uniformly
continuous, lower semicontinuous, and upper semicontinuous real
functions on the metric space $X$, respectively.\par

Given an open subset $U$ of either $\R^d$ or $\R^{d+1}$ and a
measurable function $g:U\to\R$, we  write
$\|g\|_{L^\infty(U)}$, or simply $\|g\|_\infty$ when no ambiguity is
possible, to refer to the usual $L^\infty$-norm of $g$. The space of
essentially bounded functions on $U$ is denoted by $L^\infty(U)$.

Let $k\in\N$. We denote by $\D{C}^k(\R^d)$ the space of continuous functions that are differentiable in $\R^d$ with continuous derivatives up to order $k$ inclusively,  and { set} $\D{C}^\infty(\R^d):=\bigcap_{k\in\N} \D{C}^k(\R^d)$. 
Furthermore, for $k\geqslant 2$, we denote by $\Lip^k(\R^d)$ the space of Lipschitz functions defined on $\R^d$ that {have Lipschitz derivatives} 
up to order $k-1$ inclusively. 
In the sequel we shall often use the notation
\[
 \|D_x g\|_\infty:=\sum_{i=1}^d \|\partial_{x_i} g\|_\infty,\qquad 
 \|D^2_x g\|_\infty:=\sum_{i,j=1}^d \|\partial^2_{x_ix_j} g\|_\infty.  
\]

We shall record the following basic density result for future use.

\begin{lemma}\label{lemma density}
Let $k\in\N$. The space of functions $\D{C}^\infty(\R^d)\cap \Lip^k(\R^d)$ is dense in $\D{UC}(\R^d)$ with respect to the $\|\cdot\|_{L^\infty(\R^d)}$ norm.
\end{lemma}

\begin{proof}
Since $\Lip(\R^d)$ is dense in $\D{UC}(\R^d)$, see for instance \cite[Theorem 1]{GaJa}, it is enough to show that any $g\in \Lip(\R^d)$ can be uniformly approximated in $\R^d$ by functions in $\D{C}^\infty(\R^d)\cap \Lip^k(\R^d)$. But this readily follows by regularizing $g$ via a convolution with a standard mollification kernel. 
\end{proof}

%

\indent Throughout the paper, we  denote by $A(x)$ a positive semi-definite symmetric $d\times d$ matrix, depending on $x\in\R^d$, with bounded and Lipschitz square root, namely, $A=\sigma^{T}\sigma$ for some $\sigma:\R^d\to\R^{m\times d}$, where $\sigma$ satisfies the following  conditions: there is a constant $\Lambda_A>0$ such that
\begin{itemize}
\item[(A1)] \quad $|\sigma(x)|\leqslant \Lambda_A$ \quad for every $x\in\R^d$;\smallskip
\item[(A2)] \quad $|\sigma(x)-\sigma(y)|\leqslant \Lambda_A |x-y|$ \quad for every $x,y\in\R^d$.\medskip
\end{itemize}
We stress that the case $A\equiv 0$ is included. We  {let} $H:\R^d\times\R^d\to \R$ {be} a
continuous function, hereafter called {\em Hamiltonian}, satisfying
the following assumptions:
\begin{itemize}
\item[(H1)] $H\in\D{UC}(\R^d\times B_R)$ for every 
    $R>0$;\smallskip
  \item[(H2)] there exist two continuous, coercive, and nondecreasing
    functions $\alpha, \beta:\R_{+}\to\R$ such that
\[
\alpha(|p|)\leqslant H(x,p)\leqslant\beta(|p|)\qquad\hbox{for every $(x,p)\in\R^d\times\R^d$}.
\]
\end{itemize}
By coercive we mean that $\displaystyle\lim_{R\to+\infty}{\alpha(R)}=\displaystyle\lim_{R\to+\infty}{\beta(R)}=+\infty$. 

Assumption (H2) amounts to saying that the Hamiltonian is coercive and locally bounded in $p$, uniformly with respect to $x$.

A Hamiltonian $H$ will be termed {\em convex} if $H(x,\cdot)$ is convex on $\R^d$ for every $x\in\R^d$, 
 and non-convex otherwise.
\smallskip 

Let us consider the unscaled equation 
\begin{equation}\label{app hj}\tag{HJ$_{1}$}
{\partial_t u}-\D{tr}(A(x)D_x^2u)+H(x, D_x u)=0\quad \hbox{in $ (0,+\infty)\times \R^d$.}
\end{equation}

We shall say that a function $v\in\D{USC}((0,+\infty)\times\R^d)$ is an (upper semicontinuous) {\em viscosity subsolution} of \eqref{app hj} if, for every  $\phi\in\D{C}^2((0,+\infty)\times \R^d)$ such that $v-\phi$ attains a local maximum at $(t_0,x_0)\in (0,+\infty)\times \R^d$, we have 
\begin{equation}\label{app subsolution test}
\partial_t \phi(t_0,x_0)-\Tr\big(A(x_0)D_x^2\phi(t_0,x_0)\big)+H(x_0, D_x \phi(t_0,x_0))\leqslant 0. 
\end{equation}
Any such test function $\phi$ will be called {\em supertangent} to $v$ at $(t_0,x_0)$. 

We shall say that $w\in\D{LSC}((0,+\infty)\times\R^d)$ is a (lower semicontinuous) {\em viscosity supersolution} of \eqref{app hj} if, for every $\phi\in\D{C}^2((0,+\infty)\times \R^d)$ such that $w-\phi$ attains a local minimum at $(t_0,x_0)\in (0,+\infty)\times \R^d$, we have 
\begin{equation}\label{app supersolution test}
\partial_t \phi(t_0,x_0)-\Tr\big(A(x_0)D_x^2\phi(t_0,x_0)\big)+H(x_0, D_x \phi(t_0,x_0))\geqslant 0. 
\end{equation}
Any such test function $\phi$ will be called {\em subtangent} to $w$ at $(t_0,x_0)$. 
A continuous function on $\cyl$ is a {\em viscosity solution} of \eqref{app hj} if it is both a viscosity sub and supersolution. Solutions, subsolutions, and supersolutions will be always intended in the viscosity sense, hence the term {\em viscosity} will be omitted in the sequel. 

It is well known, see for instance \cite{barles_book}, that the notions of viscosity sub and supersolutions are local, in the sense that the test function $\phi$ needs to be defined only in a neighborhood of the point $(t_0,x_0)$. Moreover, up to adding to $\phi$ a superquadratic term, such a point can be always assumed to be either a strict local maximum or a strict local minimum point of $u-\phi$. If this case we shall say that $\phi$ is a {\em strict supertangent} (resp., {\em strict subtangent}) to $u$ at $(t_0,x_0)$.\par

We shall denote by $\K$ the space of functions $u: \ccyl\to\R$ for which there exists a function $f\in\D{UC}(\R^d)$, depending on $u$, such that, for every $T>0$, 
\begin{equation}\label{def K}
|u(t,x)-f(x)|\leqslant C_T\quad\hbox{for all $(t,x)\in\cTcyl$}
\end{equation}
for some constant $C_T>0$. We {let}  $\K_*$ {be} the subspace of functions $u\in\K$ which
  satisfy the following uniform continuity condition in time at $t=0$:
\begin{align}\label{def K*}
&\qquad\qquad\hbox{for every $ a >0$ there exists $M_ a >0$ such that}&\nonumber\\
&\qquad |u(t,x)-u(0,x)|\leqslant  a +t\,M_ a \qquad\hbox{for every $(t,x)\in\ccyl$.}\tag{*}
\end{align}

\begin{definizione}\label{def well posed}
We shall say that the Cauchy problem for \eqref{app hj} is {\em $\K_*$-well-posed} if the following two properties hold:\smallskip
\begin{itemize}
 \item[(a)] (Existence) for every $g\in\D{UC}(\R^d)$, there exists a continuous function $u\in\K_*$ which solves \eqref{app hj} and satisfies the initial condition $u(0,\cdot)=g$ on $\R^d$;\medskip
 \item[(b)] (Comparison) if $u_1,\,u_2$ are continuous solutions to \eqref{app hj} belonging to $\K_*$ with $u_1(0,\cdot),\,u_2(0,\cdot)\in\D{UC}(\R^d)$, then 
 \[
  \|u_1(t,\cdot)-u_2(t,\cdot)\|_{L^\infty(\R^d)}\leqslant \|u_1(0,\cdot)-u_2(0,\cdot)\|_{L^\infty(\R^d)}\quad\hbox{for all $t\in[0,+\infty)$.}
 \]
\end{itemize}
\end{definizione}
\smallskip We start by proving a comparison principle, 
which will be used several times throughout the paper.

\begin{prop}\label{prop general Lip comparison}
Assume that $A$ satisfies (A1)-(A2) and  $H$ satisfies (H1). Let $v\in\D{USC}(\ccyl)$ and $w\in\D{LSC}(\ccyl)$ be, respectively, a sub and a supersolution of \eqref{app hj} belonging to $\K$. Let us furthermore assume that, for every $T>0$, either $D_x v$ or $D_x w$ belongs to $\left(L^\infty(\cTcyl)\right)^d$. Then 
\[
v(t,x)-w(t,x)\leqslant \sup_{\,\R^d}\big(v(0,\cdot)-w(0,\cdot)\big)\quad\hbox{for every  $(t,x)\in\ccyl$.}  
\]
\end{prop}

\begin{proof} 
Excluding trivial cases and adding a constant to $w$ as necessary, we can assume that $\sup_{\,\R^d}\big(v(0,\cdot)-w(0,\cdot)\big)=0$. Our task  is thus reduced to proving that $v\leqslant w$ on $\ccyl$. By definition of $\K$ and Lemma \ref{lemma density}, there exists a function $g\in\D{C}^\infty(\R^d)\cap \Lip^3(\R^d)$ 
such that $v-g$ is bounded in $\cTcyl$, for every fixed $T>0$. Now notice that the functions $\tilde v(t,x):=v(t,x)-g(x)$ and $\tilde w(t,x):=w(t,x)-g(x)$ are, respectively, an upper semicontinuous subsolution and a lower semicontinuous supersolution of \eqref{app hj} with modified Hamiltonian $\tilde H(x,p)=-\D{tr}(A(x)D^2_x g)+H(x,p+D_x g)$. Therefore, it suffices to establish the result for $\tilde v$, $\tilde w$ and $\
\tilde H$ in place of $v,\,w$ and $H$, respectively. Notice that, for every fixed $T>0$, the function $\tilde v$ is bounded in $\cTcyl$, while $\tilde w$ satisfies
\[
 \tilde w(t,x)\geqslant -2C_T+\tilde w(0,x)\geqslant -2C_T+\tilde v(0,x)\quad\hbox{for all $(t,x)\in\cTcyl$}
\]
for some constant $C_T>0$, by definition of $\K$. The assertion {now} follows from Proposition 1.4 in \cite{D16} with $U:=\R^d$. 
\end{proof}

If $A\equiv 0$ and $H$ satisfies (H1)-(H2), then the Cauchy problem
for (HJ$_1$) is always $\K_*$-well-posed, as stated in Theorem~\ref{teo
  well posed first order} below. But first we record a slightly more
general comparison result than (b) for future use.

\begin{prop}\label{prop first order comparison}
Let $A\equiv 0$ and $H$ satisfy hypotheses (H1)-(H2). Suppose that $v\in\D{USC}(\ccyl)$ and $w\in\D{LSC}(\ccyl)$ are, respectively, a sub and a supersolution of \eqref{app hj} belonging to $\K_*$. Then 
\[
w(t,x)-v(t,x)\leqslant \sup_{\,\R^d}\big(w(0,\cdot)-v(0,\cdot)\big)\quad\hbox{for every  $(t,x)\in\ccyl$.}  
\] 
\end{prop}

A proof can be found in \cite[Proposition A.2]{DZ15}. 

\begin{teorema}\label{teo well posed first order}
  Let $A\equiv 0$ and $H$ satisfy hypotheses (H1)-(H2). Then the
  Cauchy problem for (HJ$_1$) is $\K_*$-well-posed. Moreover, for each
  $g\in\D{UC}(\cyl)$, the solution $u$ belongs to
  $\D{UC}(\ccyl)$. 
  Furthermore, if $g\in \Lip(\R^d)$, the solution $u$ is
  Lipschitz continuous in $\ccyl$ and its Lipschitz constant depends
  only on $\|Dg\|_{L^\infty(\R^d)}$ and the functions
  $\alpha,\,\beta$.
\end{teorema}

\begin{proof}
Let us first assume $g\in \Lip(\R^d)$. Take a function 
$f\in\D{C}^1(\R^d)\cap \Lip(\R^d)$  such that $\|g-f\|_\infty<1$ and a constant $C$ satisfying 
\begin{equation*}
C>\max\{-\alpha(\|Dg\|_\infty), \beta(\|Dg\|_\infty)\}. 
\end{equation*}
Notice that, in particular, $C>\|H(x,Dg(x))\|_\infty$. 
Choose $n\in\N$ large enough so that the Hamiltonian
\[
H_n(x,p):=\max\left\{H(x,p),2\beta(|p|)-n \right\},\qquad\hbox{$(x,p)\in\R^d\times\R^d$}
\]
satisfies
\begin{eqnarray}\label{app modified Ham}
H_n=H\qquad \hbox{on $U:=\big\{(x,p)\in\R^d\times\R^d:\, H_n(x,p)< 3C\,\big\}$}.
\end{eqnarray}
The modified Hamiltonian $H_n$ satisfies the additional condition (H1) in \cite[p.\,64]{barles2}, thus
we can apply  \cite[Theorem 8.2]{barles2} and infer the existence of a function $\tilde u\in\Lip(\ccyl)$ which solves (HJ$_1$) in 
$\ccyl$ with $\widetilde H(x,p)=H_n(x,p+D_x f)$ in
place of $H$ and initial condition $\tilde u(0,\cdot)=g-f$ on $\R^d$. Moreover, it is not hard to see that 
the functions $g(x)-f(x)-2Ct$ and $g(x)-f(x)+2Ct$ are, respectively, a sub and a supersolution of (HJ$_1$) in 
$\ccyl$ with $\widetilde H$ in place of $H$. By arguing as in the proof of \cite[Theorem 8.2]{barles2} we obtain 
$\|\partial_t \tilde u\|_\infty\leqslant 2C$, in particular  
\begin{equation}\label{uff}
 H_n(x,D_x \tilde u(t,x)+Df(x))\leqslant 2C\qquad\hbox{for a.e. $(t,x)\in\cyl$}.
\end{equation}
Then the function $u:=\tilde u+f$ belongs to $\Lip(\ccyl)$ and solves
(HJ$_1$) with $H_n$ in place of $H$ and initial condition
$u(0,\cdot)=g$ on $\R^d$. It is a standard fact that \eqref{uff}
implies that $(x,D_x\phi(t,x))\in U$ for any sub or supertangent
$\phi$ to $u$ at $(t,x)\in\cyl$. In view of \eqref{app modified Ham},
we finally infer that $u$ is a solution of (HJ$_1$) in $\ccyl$ as
well.

Let now assume that $g\in\D{UC}(\R^d)$. Then by Lemma~\ref{lemma density} there exists a sequence of functions $g_n\in \Lip(\R^N)$ uniformly converging to $g$ on $\R^d$. Let us denote by $u_n$ the corresponding Lipschitz solution to \eqref{app hj} with initial datum $g_n$. 
By Proposition \ref{prop general Lip comparison}  we have
\[
\|u_m-u_n\|_{L^\infty(\ccyl)}
\leqslant
\|g_m-g_n\|_{L^\infty(\R^d)},
\]
that is, $(u_n)_n$ is a Cauchy sequence in $\ccyl$ with respect to the sup-norm. Hence the Lipschitz continuous functions $u_n$ uniformly converge to a function $u$ on $\ccyl$, which is therefore uniformly continuous. By the stability of the notion of viscosity solution, we conclude that $u$ is a solution of \eqref{app hj} with initial datum $g$. The remainder of the assertion is a straightforward consequence of Proposition \ref{prop first order comparison}.
\end{proof}

In the case $A\not\equiv 0$, we need to introduce additional assumptions on the Hamiltonian 
to be sure that the Cauchy problem for the viscous equation \eqref{app hj} is
$\K_*$-well-posed. Our examples in Section~\ref{sez stationary ergodic
  homogenization} will satisfy these conditions.
\begin{definizione}\label{Fam}
  We shall say that a function $H\in \D{C}(\R^d\times\R^d)$ belongs to
  the class $\Ham$ for some constants $\alpha_0,\beta_0>0$ and $\gamma>1$
  if it satisfies the following inequalities:
\begin{itemize}
\item[(i)] \quad $\alpha_0|p|^\gamma-1/\alpha_0\leqslant H(x,p)\leqslant\beta_0(|p|^\gamma+1)$\ \ for all $x,p\in\R^d$;\medskip
\item[(ii)] \quad $|H(x,p)-H(y,p)|\leqslant\beta_0(|p|^\gamma+1)|x-y|$\ \ for all $x,y,p\in\R^d$;\medskip
\item[(iii)] \quad
  $|H(x,p)-H(x,q)|\leqslant\beta_0\left(|p|+|q|+1\right)^{\gamma-1}|p-q|$\ \  for all $x,p,q\in\R^d$.\medskip
\end{itemize}
\end{definizione}
Clearly, any $H\in\Ham$ satisfies (H1)-(H2). The following comparison principle holds:

\begin{prop}\label{prop comp in Ham}
\begin{sloppypar}
Assume $A$ satisfies (A1)-(A2) and $H\in\Ham$. Let $v\in\D{USC}(\ccyl)$ and $w\in\D{LSC}(\ccyl)$ be, respectively, a sub and a supersolution of \eqref{app hj} belonging to $\K$ and such that either $v(0,\cdot)$ or $w(0,\cdot)$ is in $\D{UC}(\R^d)$.  Then \vspace{-3ex}
\end{sloppypar}
\[
v(t,x)-w(t,x)\leqslant \sup_{\,\R^d}\big(v(0,\cdot)-w(0,\cdot)\big)\quad\hbox{for every  $(t,x)\in\ccyl$.}  
\]
\end{prop}

The assertion follows by arguing as in the proof of Proposition \ref{prop general Lip comparison} and by 
using  \cite[Theorem 3.1]{D16} in place of \cite[Proposition 1.4]{D16}.

From \cite{D16} we infer the following result.

\begin{teorema}\label{teo main HJ inf H}
  Let $A$ satisfy (A1)-(A2) and $H\in\Ham$. Then the 
  Cauchy problem for \eqref{app hj} is $\K_*$-well-posed. Moreover,
  for every $g\in\D{UC}(\R^d)$, the solution $u$ belongs to
  $\D{UC}(\ccyl)$. If
  $g\in\Lip^2(\R^d)$, then the solution $u$ is
  Lipschitz continuous on $\ccyl$.
\end{teorema}

\begin{proof}
 Given $g\in \D{UC}(\R^d)$, we pick $f\in\D{C}^\infty(\R^d)\cap {\Lip^3}(\R^d)$ such that $\|g-f\|_\infty<1$. 
According to Theorem 3.2 in \cite{D16}, there exists $\tilde u\in\D{UC}(\ccyl)$ that solves \eqref{app hj} with modified Hamiltonian $\tilde H(x,p)=-\D{tr}(A(x)D^2_x f)+H(x,p+D_x f)$ and initial datum $g-f$. Moreover, $\tilde u$ is bounded in cylinders of the form $\cTcyl$, for every fixed $T>0$. It is easily seen that $u:=\tilde u+ f$ is in $\D{UC}(\ccyl)\cap\K\subset\K_*$ and solves \eqref{app hj} with initial condition $u(0,\cdot)=g$ on $\R^d$.

If $g\in \Lip^2(\R^d)$, then we can ensure that $f$ also satisfies the inequalities
\begin{equation}\label{estimates derivatives f}
\|Df-Dg\|_\infty< 1,\qquad \|D^2f-D^2g\|_\infty< 2\|D^2g\|_\infty.
\end{equation}
In view of Theorem 3.2 in \cite{D16}, we conclude that $\tilde u$, and hence $u$, are Lipschitz continuous in $\ccyl$. 
The remainder of the assertion is a straightforward consequence of Proposition \ref{prop comp in Ham}.
\end{proof}

\begin{oss}\label{oss Lip dependence}
From Theorem 3.2 in \cite{D16} and in view of \eqref{estimates derivatives f}, we can also infer that the 
Lipschitz constant $\kappa$ of the solution $u$ with initial datum $g\in{\Lip^2}(\R^d)$ is a locally bounded function of $\|Dg\|_{L^\infty(\R^d)}$ and $\|D^2 g\|_{L^\infty(\R^d)}$.
More precisely,  
$\kappa=\widetilde\kappa\left(\|Dg\|_{L^\infty(\R^d)}\right.,$ $\left.\|D^2 g\|_{L^\infty(\R^d)}\right)$ for some function $\widetilde k:[0,+\infty)\times[0,+\infty)\to\R$ which only depends on $\Lambda_A$, $\alpha_0,\,\beta_0$ and $\gamma>1$ and is locally bounded in its arguments. This remark will be needed in Section \ref{sez non convex homogenization}. 
\end{oss}

\section{Homogenization: from linear to general initial data}\label{sez general}

In this section we consider viscous and non-viscous Hamilton-Jacobi
equations of the form (HJ$_\epsilon$) where the matrix $A$ and the
Hamiltonian $H$ satisfy (A1)-(A2) and (H1)-(H2) respectively. We
shall also assume that the Cauchy problem for \eqref{app hj} is
$\K_*$-well-posed, in the sense of Definition \ref{def well posed}.
Note that this is equivalent to the $\K_*$-well-posedness of the
Cauchy problem for (HJ$_\epsilon$), for some (and, thus,
  for all) $\eps>0$. It suffices to remark that $u\in\K_*$ is a
continuous solution to \eqref{app hj} if and only if the function
$u^\eps(t,x):=\epsilon u(t/\epsilon,x/\epsilon)$ is a continuous
solution to \eqref{intro viscous hj} which belongs to $\K_*$.

We are interested in the asymptotic behavior of the solutions  to \eqref{intro viscous hj} when $\eps\to 0^+$. We would like to show that, in order to establish a homogenization result for \eqref{intro viscous hj}, it is sufficient to consider only linear initial data. 
%
%
%
%
To this aim, for every fixed $\theta\in\R^d$ and $\eps>0$, we denote by $u_\theta$ and $u^\eps_\theta$ the unique continuous functions in $\K_*$ that solve  \eqref{app hj} and \eqref{intro viscous hj} respectively subject to the initial condition $u_\theta(0,x)=u^\eps_\theta(0,x)=\langle \theta,x\rangle$. 

The main result of this section is the following theorem. 
\begin{teorema}\label{app teo hom}
  Let $A$ and $H$ satisfy hypotheses (A1)-(A2) and (H1)-(H2),
  respectively. Assume that the Cauchy problem for \eqref{app hj} is
  $\K_*$-well-posed, and that either one of the following conditions
  is satisfied:\smallskip
\begin{itemize}
\item[(H3)] there exists a modulus of continuity $m$ such that 
\[
|H(x,p_{1})-H(x,p_2)|\leqslant m(|p_1-p_2|)\qquad\hbox{for all $x\in\R^d$\ and\ $p_1,p_2\in\R^d$;\medskip}
\]
\item[(L)\ \ ] for every $\theta\in\R^d$, there exists a constant $\kappa=\kappa(\theta)$ such that 
\[
|u_\theta(t,x)-u_\theta(t,y)|\leqslant \kappa |x-y|\qquad\hbox{for all $x,y\in\R^{d}$ and $t\geqslant 0$.}
\]
\end{itemize}
Finally, suppose that there exists a continuous and coercive Hamiltonian $\overline H:\R^d\to\R$ such that, for every $\theta\in\R^d$
\begin{equation}\label{app hom linear data}
u^\eps_\theta(t,x)\ {\ucv}\ \langle \theta,x\rangle-t\overline H(\theta)\quad\hbox{in $\ccyl$ as $\epsilon\to 0$.}
\end{equation}
Then, for every  $g\in\D{UC}(\R^d)$, the unique
continuous function $u^\eps$ in $\K_*$ solving \eqref{intro viscous
  hj} with initial
condition $u^\eps(0,\cdot)=g$ converges, locally uniformly in $\ccyl$  
as $\eps\to 0^+$, to the unique solution $\overline u\in\D{UC}(\ccyl)$
of
\begin{equation}\label{app effective eq}
\partial_t\overline u+\overline H(D_x\overline u)=0\qquad\hbox{in $\cyl$}
\end{equation}
with the initial condition $\overline u(0,\cdot)=g$.
\end{teorema}

\begin{oss}\label{oss rescaling}
  It is easy to see that, by uniqueness,
  $u^\eps_\theta(t,x)=\eps\, u_\theta(t/\eps,x/\eps)$. Therefore, the
  hypothesis (L) amounts to requiring that the functions \
  $\{u^\eps_\theta(t,\cdot)\,:\,0<\eps\leqslant 1,\, t\geqslant 0\,\}$\ are
  equi-Lipschitz in $\R^d$. A similar remark applies to condition
  (L$'$) below.
\end{oss}

To keep the proof of Theorem \ref{app teo hom} concise, we shall first
prove the following fact.

\begin{prop}\label{app prop hom}
  Let us assume that all the hypotheses of Theorem \ref{app teo hom}
  are in force. Let $g\in\D{UC}(\R^d)$ and, for every $\eps>0$, let us
  denote by $u^\eps$ the unique continuous function in $\K_*$ that solves \eqref{intro viscous hj}
  subject to the initial condition
  $u^\eps(0,\cdot)=g$. Set
\begin{eqnarray*}
u^*(t,x)&:=&\lim_{r\to 0}\ \sup\{ u^{\eps}(s,y)\,:\, (s,y)\in (t-r,t+r)\times B_r(x),\ 0<\eps<r\,\},\\
u_*(t,x)&:=&\lim_{r\to 0}\ \inf\{ u^\eps(s,y)\,:\, (s,y)\in (t-r,t+r)\times B_r(x),\ 0<\eps<r\,\}.
\end{eqnarray*}
Let us assume that $u^*$ and $u_*$ are finite valued. Then\smallskip
\begin{itemize}
\item[(i)] $u^*\in\D{USC}(\ccyl)$ and it is a viscosity subsolution of \eqref{app effective eq};\medskip
\item[(ii)] $u_*\in\D{LSC}(\ccyl)$ and it is a viscosity supersolution of \eqref{app effective eq}.\\
\end{itemize}
\end{prop}

Theorem \ref{app teo hom}  follows {readily} from Proposition \ref{app prop hom}, as we show
now.

\begin{proof}[Proof of Theorem \ref{app teo hom}]
Let us first assume $g\in {\D{C}^2(\R^d)\cap\Lip^2}(\R^d)$. Take a constant $M$ large enough so that 
\[
M> \|\D{tr}(A(x)D_x^2 g(x))   \|_\infty + \|H(x,D g(x))\|_\infty.
\]
Then the functions $u_-(t,x):=g(x)-Mt$ and $u_+(t,x):=g(x)+Mt$ are,
respectively, a Lipschitz continuous sub and supersolution of
\eqref{intro viscous hj} for every $0<\eps\leqslant 1$. By Proposition \ref{prop general Lip comparison}, we get $u_-\leqslant u^\eps \leqslant u_+$ in $\ccyl$ for
every $0<\eps\leqslant 1$. By the definition of relaxed semilimits we
infer
\[u_-(t,x)\leqslant u_*(t,x)\leqslant u^*(t,x)\leqslant u_+(t,x) \quad
\text{for all $(t,x)\in\ccyl$},\]
in particular, $u_*$, $u^*$ satisfy
$u_*(0,\cdot)=u^*(0,\cdot)=g$ on $\R^d$ and  belong to $\K_*$.  By Proposition \ref{app prop
  hom}, we know that $u^*$ and $u_*$ are, respectively, an upper
semicontinuous subsolution and a lower semicontinuous supersolution of
the effective equation \eqref{app effective eq}. We can therefore
apply the Comparison Principle for \eqref{app effective eq} stated in Proposition \ref{prop first order comparison}
 to obtain $u^*\leqslant u_*$ on
$\ccyl$. Since the opposite inequality holds by the definition of upper and
lower relaxed semilimits, we conclude that the function
\[
\overline u(t,x):=u_*(t,x)=u^*(t,x)\qquad\hbox{for all $(t,x)\in\ccyl$}
\]
is the unique continuous viscosity solution of \eqref{app effective
  eq} in $\K_*$ such that $\overline u(0,\cdot)=g$ on
$\R^d$. Furthermore, by Theorem \ref{teo well posed first order}, $\overline u$ is Lipschitz continuous on $\ccyl$. The fact that the relaxed
semilimits coincide implies that  $u^\eps$ converge locally
uniformly in $\ccyl$ to $\overline u$, see for instance \cite[Lemma
6.2, p. 80]{ABIL}.

When the initial datum $g$ is just uniformly continuous on $\ccyl$, the
result easily follows from the above and Lemma \ref{lemma density} by
approximating $g$ with a sequence of initial data belonging to
${\D{C}^2(\R^d)\cap\Lip^2}(\R^d)$ and  by
making use of the contraction property (b) of Definition \ref{def well posed} for \eqref{intro viscous hj}.
\end{proof}

\begin{proof}[Proof of Proposition  \ref{app prop hom}]
  The fact that $u^*$ and $u_*$ are upper and lower semicontinuous on
  $\ccyl$ is an immediate consequence of their definition. Let us
  prove (i), i.e.\ that $u^*$ is a subsolution of \eqref{app effective
    eq}. The proof of (ii) is analogous.

  We make use of Evans's perturbed test function method, see
  \cite{Evans89}. Let us assume, by contradiction, that $u^*$ is not a
  subsolution of \eqref{app effective eq}. Then there exists a
  function $\phi\in\D{C}^2(\cyl)$ that is a strict supertangent of
  $u^*$ at some point $(t_0,x_0)\in\cyl$ and for which the subsolution
  test fails, i.e.
\begin{equation}\label{app test fails}
\partial_t\phi(t_0,x_0)+\overline H(D_x\phi(t_0,x_0))>3\delta
\end{equation}
for some $\delta>0$. For $r>0$ define
$V_r:=(t_0-r,t_0+r)\times B_r(x_0)$. Choose $r_0>0$ to be small enough so
that $V_{r_0}$ is compactly contained in $\cyl$ and $u^*-\phi$ attains a strict
local maximum at $(t_0,x_0)$ in $V_{r_0}$. In particular, we have for
every $r\in(0,r_0)$
\begin{equation}\label{app strict supertangent}
\max_{\partial V_r} (u^*-\phi)<\max_{\overline V_r} (u^*-\phi)=(u^*-\phi)(t_0,x_0).
\end{equation}
Let us set $\theta:=D_x\phi(t_0,x_0)$ and for every $\eps>0$ denote by
$u^\eps_\theta$ the unique continuous function in $\K_*$ that solves \eqref{intro viscous
  hj} subject to the initial condition
$u^\eps_\theta(0,x)=\langle\theta,x\rangle$. We claim that {there is an} $r\in(0,r_0)$ such that the function
\[
\phi^\eps(t,x):=\phi(t,x)+u_\theta^\eps(t,x)-\left(\langle\theta,x\rangle-t\overline H(\theta)\right)
\]
is a supersolution of \eqref{intro viscous hj} in $V_r$ for every $\eps>0$ small enough. Indeed, by a direct computation we first get 
\begin{eqnarray}\label{app ineq1}
&&\partial_t\phi^\eps-\eps\,\Tr\left(A\left(\frac{x}{\eps}\right) D^2_x\phi^\eps\right)+H\left(\frac{x}{\eps},D_x\phi^\eps\right)
=
\partial_t\phi+\overline H(\theta)-\eps\,\Tr\left(A\left(\frac{x}{\eps}\right) D^2_x\phi\right)\nonumber\medskip\\
&&+\partial_t u^\eps_\theta-\eps\,\Tr\left(A\left(\frac{x}{\eps}\right)D_{x}^2u^\eps_\theta\right)+H\left(\frac{x}{\eps},D_xu^\eps_\theta+D_x\phi-\theta\right)
\end{eqnarray}
in the viscosity sense in $V_r$. Using \eqref{app test fails},
assumption (A1), and the fact that $\phi$ is of class $C^2$, we get
that there is an $r\in (0,r_0)$ such that for all sufficiently small
$\eps>0$ and all $(t,x)\in V_r$
\[
\partial_t\phi(t,x)+\overline
H(\theta)-\eps\,\Tr\left(A\left(\frac{x}{\eps}\right)
  D^2_x\phi\right)>3\delta-\eps\,\Tr\left(A\left(\frac{x}{\eps}\right)
  D^2_x\phi\right)>2\delta.
\]
Moreover, by taking into account either (H3) or (L) (together with Remark \ref{oss rescaling}) and (H1), we can further reduce $r$ if necessary to get 
\[
H\left(\frac{x}{\eps},Du^\eps_\theta+D_x\phi-\theta\right)
>
H\left(\frac{x}{\eps},Du^\eps_\theta\right)-\delta\qquad\hbox{in the viscosity sense in $V_r$.}
\]
Plugging these relations into \eqref{app ineq1} and using the fact that $u^\eps_\theta$ is a solution of \eqref{intro viscous hj}, we finally get 
\begin{eqnarray*}
&&\partial_t\phi^\eps-\eps \Tr\left(A\left(\frac{x}{\eps}\right)D_x^2\phi^\eps\right)+H\left(\frac{x}{\eps},D_x\phi^\eps\right)\\
&&\qquad\qquad\qquad\qquad>
\delta+\partial_t u^\eps_\theta-\Tr\left(A\left(\frac{x}{\eps}\right)D_x^2u^\eps_\theta\right)+H\left(\frac{x}{\eps},Du^\eps_\theta\right)
=\delta>0
\end{eqnarray*}
in the viscosity sense in $V_r$, thus showing that $\phi^\eps$ is a supersolution of \eqref{intro viscous hj} in $V_r$. We now need a comparison principle for equation \eqref{intro viscous hj} in $V_r$ applied to $\phi^\eps$ and $u^\eps$ to infer that 
\[
\sup_{ V_r} (u^\eps-\phi^\eps){\leqslant}\max_{\partial V_r} (u^\eps-\phi^\eps).
\]
If condition (H3) holds, we can apply \cite[Theorem 3.3 and Section
5.C]{users}. If (L) is satisfied, then $D_x \phi^\eps\in L^\infty (V_r)$ and a standard 
argument (see, for instance, Proposition 1.4 in \cite{D16}) shows that the comparison principle for \eqref{intro viscous
  hj} in $V_r$ holds in this case as well.  Now notice that, by the
assumption \eqref{app hom linear data}, $\phi^\eps\ucv\phi$ in
$\overline V_r$. Taking the limsup of the last inequality as
$\eps\to 0^+$ we obtain
\[
\sup_{V_r} (u^*-\phi)
\leqslant 
\varlimsup_{\eps\to 0^+} \sup_{ V_r} (u^\eps-\phi^\eps)
\leqslant
\varlimsup_{\eps\to 0^+}
\max_{\partial V_r} (u^\eps-\phi^\eps)
\leqslant
\max_{\partial V_r} (u^*-\phi),
\]
a contradiction with \eqref{app strict supertangent}. This proves that
$u^*$ is a subsolution of \eqref{app effective eq}. 
\end{proof}

Notice that in Theorem \ref{app teo hom} and Proposition \ref{app prop hom} we have assumed {\em as a hypothesis} that $\overline H:\R^d\to\R$ is continuous and coercive. While the latter property is inherited from that of $H$, 
as we show below, the 
continuity is not guaranteed a priori, but is deduced from the way the effective Hamiltonian is obtained. 
Our next result shows that this property can be, for instance, deduced when the bounds on the derivatives of the functions $u^\eps_\theta$ with respect to $x$ are locally uniform with respect to $\theta$. 

\begin{prop}\label{app prop useful}
  Let $A$ and $H$ satisfy hypotheses (A1)-(A2) and (H1)-(H2),
  respectively, and the Cauchy problem for \eqref{app hj} be
  $\K_*$-well-posed. Assume that, for every $\theta\in\R^d$, the
  convergence stated in \eqref{app hom linear data} holds for some
  function $\overline H:\R^d\to\R$. Then $\overline H$ satisfies
  assumption (H2).  Let us furthermore assume either condition (H3)
  or the following:
\begin{itemize}
 \item[(L$'$)] for every $r>0$, there exists a constant $\kappa_r$ such that 
\[
|u_\theta(t,x)-u_\theta(t,y)|\leqslant \kappa_r |x-y|\qquad\hbox{for all $x,y\in\R^{d}$, $t\geqslant 0$ and $\theta\in B_r$.}
\]

\end{itemize}
Then, for every $r>0$, there exists a continuity modulus  $m_r$ such that 
\begin{equation}\label{claim continuity effective H}
\left| u^\eps_{\theta_1}(t,x)-u^\eps_{\theta_2}(t,x) \right|\leqslant t\,m_r\left(|\theta_1-\theta_2|\right)+|x||\theta_1-\theta_{2}|
\end{equation}
{for every $x\in\R^d$ and $\theta_1,\theta_2\in B_r$.}
In particular, the effective Hamiltonian $\overline H$ is continuous.
\end{prop}

\begin{proof}
  According to \eqref{app hom linear data}, the effective Hamiltonian
  is defined by the following formula:
\begin{equation}\label{app def effective H}
\overline H(\theta)=\lim_{\eps\to 0^+} -u^\eps_\theta(1,0)\qquad\hbox{for every $\theta\in\R^d$.}
\end{equation}
Let us first show that $\overline H$ satisfies (H2). Let us fix
$\theta\in\R^d$. It is easily seen that the functions
$u_-(t,x)=\langle \theta, x\rangle-t\beta(|\theta|)$ and
$u_+(t,x)=\langle \theta, x\rangle-t\alpha(|\theta|)$ are
classical sub and supersolutions to \eqref{intro viscous hj} for every
$\eps>0$, respectively. By Proposition \ref{prop general Lip comparison} we get
$u_-\leqslant u^\eps_\theta\leqslant u_+$ in $\ccyl$, and the assertion immediately follows from this in
view of \eqref{app def effective H}. 

To prove \eqref{claim continuity effective H}, we let
$v_\theta^\eps(t,x):=u^\eps_{\theta}(t,x)-\langle\theta,x\rangle$ for
every $\theta\in\R^d$. Then $v^\eps_\theta$ is a solution of
\begin{eqnarray}\label{app continuity H effective}
\displaystyle\partial_t v^\eps_\theta
-
\eps\,\Tr\left(A\left(\frac{x}{\eps}\right)D_x^2v^\eps_\theta\right)
+
H\left(\frac{x}{\eps},\theta+D_x v^\eps_\theta\right)=0 \qquad \hbox{in $\cyl$}
\end{eqnarray}
satisfying $v^\eps_\theta(0,x)=0$ in $\R^d$. If hypothesis (L$'$) is
in force, then, in view of Remark \ref{oss rescaling}, for every $r>0$
there exists $\rho(r)>0$ such that
$\|D_x v^\eps_\theta\|_\infty<\rho(r)$ for every $0<\eps\leqslant1$ and
$\theta\in B_r$.  Let us fix $r>0$ and denote by $m_r$ a modulus of
continuity such that
\[
|H(x,p_1)-H(x,p_2)|\leqslant m_r(|p_1-p_2|)\qquad\hbox{for all $x\in\R^d$ and $p_1,p_2\in B_{r+\rho(r)}$.}
\]
If, on the other hand, hypothesis (H3) is in force, then the above
inequality holds with $m_r$ independent of $r$.  Take
$\theta_1,\theta_2\in B_r$. Then for every $\eps>0$
\[
\left| H\left(\frac{x}{\eps},\theta_1+D_x v^\eps_{\theta_1}\right)-H\left(\frac{x}{\eps},\theta_{2}+D_x v^\eps_{\theta_1}\right) \right|\leqslant m_r\left(|\theta_1-\theta_2|\right)
\]
in the viscosity sense in $\cyl$. We infer that the functions $v^\eps_{\theta_1}(t,x)-t\,m_r\left(|\theta_1-\theta_2|\right)$ and $v^\eps_{\theta_1}(t,x)+t\,m_r\left(|\theta_1-\theta_2|\right)$ are, respectively, a sub and a supersolution of \eqref{app continuity H effective} with $\theta:=\theta_2$. 
By Proposition \ref{prop general Lip comparison} we conclude that
\[
\left| v^\eps_{\theta_1}(t,x)-v^\eps_{\theta_2}(t,x) \right|\leqslant t\, m_r\left(|\theta_1-\theta_2|\right).
\]
By the definition of $v^\eps_{\theta}$, we get \eqref{claim continuity effective H}, and, in view of  \eqref{app def effective H}, we  {obtain}, in particular,
\[
\left|\overline H(\theta_1)-\overline H(\theta_2)\right|\leqslant m_r\left(|\theta_1-\theta_2|\right)\qquad\hbox{for all $\theta_1,\theta_2\in B_r$},
\]
yielding the asserted continuity of $\overline H$. 
\end{proof}
\medskip

\section{Homogenization in the stationary ergodic setting}\label{sez stationary ergodic homogenization}

\subsection{Stationary ergodic framework.} \label{sez stationary ergodic framework}
In this section we
recall basic definitions and discuss some of the
implications of stationarity and ergodicity for the results of
Section~\ref{sez general}.

We denote by $(\Omega,\F, \PP)$ a {\em probability space}, where
$\Omega$ is a sample space, ${\cal F}$ is a sigma-algebra of subsets
of $\Omega$ and $\PP$ is a probability measure on
$(\Omega,\F)$. 
We shall denote by $\Bor(\R^d)$ the $\sigma$-algebra of Borel
subsets of $\R^d$. When defining measurability we shall always work with
measurable spaces $(\R^d,\Bor(\R^d))$, $(\Omega,{\cal F})$, and the
product space $\R^d\times\Omega$ endowed with the corresponding product
$\sigma$-algebra $\Bor
(\R^d)\otimes\F$. 
Polish spaces $\D{C}(\R^d)$ and $\D{C}(\R^d\times\R^d)$ are considered with
the topology of locally uniform convergence on $\R^d$ and
$\R^d\times\R^d$ respectively and are equipped with their Borel
$\sigma$-algebras.

A {\em d-dimensional dynamical system of shifts} $(\tau_x)_{x\in\R^d}$ is
defined as a family of mappings $\tau_x:\Omega\to\Omega$ which
satisfy the following properties:
\begin{enumerate}
\item[{\em (1)}] ({\em group property}) 
  $\tau_0=id$,\quad $\tau_{x+y}=\tau_x\comp\tau_y$;

\item[{\em (2)}] ({\em preservation of measure}) $\tau_x:\Omega\to\Omega$ is
  measurable and 
  $\PP(\tau_x E)=\PP(E)$ for every $E\in\F$;

\item[{\em (3)}] ({\em joint measurability}) the map
  $(x,\omega)\mapsto \tau_x\omega$ from ${ \R^d}\times\Omega$ to $\Omega$
  is measurable
.
\end{enumerate}
\begin{sloppypar}
The above properties guarantee (see \cite[(7.2)]{JKO}) that $(\tau_x)_{x\in\R^d}$ is continuous, i.e.\ 
${\lim_{|x|\to
  0}\|f(\tau_x\omega)-f(\omega)\|_{L^2(\Omega)}=0}$ for every
$f\in L^2(\Omega)$.
\end{sloppypar}

We make the crucial assumption that $(\tau_x)_{x\in{\R^d}}$ is {\em
  ergodic,} i.e.\ 
any measurable function $f:\Omega\to\R$ enjoying
$f(\tau_x\omega)=f(\omega)$ a.s.\ in $\Omega$, for any fixed $x\in{\R^d}$,
is almost surely constant.

We shall say that a (measurable) random field 
$H:\Omega\to \D{C}(\R^d\times\R^d)$ is {\em stationary with respect to
  the shifts $(\tau_x)_{x\in\R^d}$} if it admits the following
representation:
\begin{equation}
  \label{st}
H(x,p,\omega)=\tilde{H}(p,\tau_x\omega)\quad\text{for all $x,p\in\R^d$ and
$\omega\in\Omega$.}  
\end{equation}
for some (measurable) $\tilde{H}:\Omega\to \D{C}(\R^d)$. Note that
\eqref{st} and the group property {\em(1)} above immediately imply
that\smallskip
  \begin{itemize}
  \item[(S)] $H(x+y,p,\omega)=H(x,p,\tau_y\omega)$\quad for all $x,y,p\in\R^d$ and $\omega\in\Omega$.\smallskip
  \end{itemize}
  Since random variables $H(x,p,\cdot)$ and $H(x+y,p,\cdot)$ have the
  same distribution due to {\em(2)}, we can say informally that (S)
  expresses the desired feature of the underlying random medium: at
  different points in space the medium statistically ``looks'' the
  same. We also remark that every $H$ satisfying (S) admits a representation of
  the form \eqref{st} (see, for instance, \cite[Proposition 3.1]{DS09}). 
Stationarity of a random process $b:\Omega\to \D{C}(\R^d)$ is
  defined in the same way simply by omitting $p$ in \eqref{st}. 
  
  We are interested in homogenization for solutions
  $u^\epsilon(t,x,\omega)$ of the Cauchy problem for equation
  (HJ$_\epsilon$) with Hamiltonian $H(x,p,\omega)$ on a set of
  $\omega$ of full measure. We shall assume that all constants in the
  assumptions on $A$ and $H$, i.e.\
  $\Lambda_A,\ \alpha(\cdot),\ \beta(\cdot)$, and the moduli of
  continuity of $H$ on $\R^d\times B_R$ for each $R>0$, are
  independent of $\omega$.

  We note that one of the consequences of is
  that the locally uniform convergence \eqref{app hom linear data} of
  Theorem~\ref{app teo hom} follows from the a.s.\ convergence at
  $t=1$ (equivalently, at an arbitrary $t>0$) and $x=0$ as long as the
  family
  $\{u_\theta^\epsilon(1,\cdot,\omega),\,0<\epsilon\leqslant
  1,\omega\in\Omega\}$ is equi-continuous.
  
\begin{lemma}\label{hom st erg}
  Let $H$ and (all entries of) $A$ be stationary in the sense of the above definition and satisfy hypotheses (H1)-(H2) and (A1)-(A2),
  respectively. We shall suppose that all uniform bounds are
  independent of $\omega$. Assume that, for every fixed $\omega$, the Cauchy problem for
  \eqref{app hj} is $\K_*$-well-posed, and that, for each
  $\theta\in\R^d$, there is a modulus of continuity $m_\theta$ such
  that for all $\epsilon\in(0,1]$ and $\omega\in\Omega$
  \begin{equation}
    \label{uc}
    |u^\epsilon_\theta(1,x,\omega)-u^\epsilon_\theta(1,0,\omega)|\leqslant m_\theta(|x|)\quad\text{for all $x\in\R^d$.}
  \end{equation}
If, for each $\theta\in\R^d$, we have  
  \begin{equation}
    \label{3.1weak}
    \lim_{\epsilon\to 0+}u_\theta^\epsilon(1,0,\omega)= -\overline{H}(\theta)
  \end{equation}
  with probability 1, then the above convergence is locally uniform, i.e.\ with probability $1$
  \begin{equation}\label{claim distilled}
  u_\theta^\epsilon(t,x,\omega)\ucv
  \dprod{\theta}{x}-t\overline{H}(\theta)\quad\text{in }\ccyl.
  \end{equation}
  In particular, if the condition (L$'$) from Proposition~\ref{app
    prop useful} is satisfied with $\kappa_r$ independent of $\omega$,
  then there is a set $\hat{\Omega}\subseteq \Omega$ of full measure
  such that the last convergence takes place for all $\theta\in\R^d$
  and all $\omega\in\hat{\Omega}$, and, thus, the conclusion of
  Theorem~\ref{app teo hom} holds for all $\omega\in\hat{\Omega}$.
\end{lemma}

The proof below is based on a by-now standard argument
  which appeared in, for instance, \cite[pp.\ 1501-1502]{KRV}, \cite[pp.\
  403-404]{DS11}, \cite[Lemma~4.10]{AS12}. It was later ``distilled'' into
  an abstract lemma in \cite[Lemma 2.4]{AT15}, which is convenient for
  time-independent applications. Since $u_\theta^\epsilon$ is time
  dependent, we shall need an additional easy step.
 
 \begin{proof}
 Fix $\theta\in \R^d$ and set 
  \[
  w(t,x,\omega):=u_\theta(t,x,\omega)-\dprod{\theta}{x}+t\overline{H}(\theta).
  \]
  For any fixed $\omega\in\Omega$, the function $w(\cdot,\cdot,\omega)$ is in $\K_*$ and solves (HJ$_1$) with Hamiltonian
  $H(\cdot,\theta+\cdot,\omega)-\overline{H}(\theta)$ and zero initial datum. By stationarity of the Hamiltonian and uniqueness of the solution, we
  conclude that $w(t,\cdot,\cdot)$ is stationary in $x$, for every fixed $t\geqslant 0$. We claim that there exists a set $\Omega_\theta$ of full measure such that, for every $\omega\in\Omega_\theta$, 
\begin{equation}\label{t1}
\limsup_{\epsilon\to 0}\sup_{y\in B_R}|u^\epsilon_\theta(1,y,\omega)-\dprod{\theta}{y}+\overline{H}(\theta)|=0
\qquad
\hbox{for all $R>0$.}
\end{equation}
To prove this, it suffices to apply \cite[Lemma 2.4]{AT15} with $X_\epsilon(x,\omega):=\eps|w(1/\eps, x,\omega)|$.  Indeed, by the rescaling 
$u^\eps_\theta(t,x,\omega)=\eps u_\theta (t/\eps,x/\eps,\omega)$, claim \eqref{t1} is 
equivalent to 
\[\PP\Big(\omega\in\Omega\,:\,\forall R>0\quad
\limsup\limits_{\epsilon\to 0}\sup\limits_{y\in
  B_{R/\epsilon}}|X_\epsilon(\cdot,\omega)|=0\Big)=1.
  \]
Let us then check that 
such $X_\epsilon$ satisfy the conditions stated in the quoted lemma. The stationarity of $X_\epsilon$ follows
  from the stationarity of $v$. The a.s.\ convergence
  $\lim\limits_{\epsilon\to 0}X_\epsilon(0,\cdot)=0$ is just a
  restatement of \eqref{3.1weak}. The property
 \[
     \PP\Big(\omega\in\Omega\,:\,\forall z\in\R^d\ \ \lim\limits_{r\to
    0}\limsup\limits_{\epsilon\to
    0}\operatornamewithlimits{osc}\limits_{
    B_{r/\epsilon}(z/\epsilon)}X_\epsilon(\cdot,\omega)=0\Big)=1
    \]
  is an immediate consequence of the assumption \eqref{uc}.   
  Indeed,
  \begin{multline*}
    \operatornamewithlimits{osc}\limits_{y\in
    B_{r/\epsilon}(z/\epsilon)}X_\epsilon(y,\omega)\leqslant\epsilon\sup_{x,y\in B_{r/\epsilon}(z/\epsilon)}|w(1/\epsilon,0,\tau_x\omega)-w(1/\epsilon,y-x,\tau_x\omega)|\\ \leqslant2r|\theta|+\sup_{|y-x|\leqslant2r}|u^\epsilon_\theta(1,0,\tau_{x/\epsilon}\omega)-u^\epsilon_\theta(1,y-x,\tau_{x/\epsilon}\omega)|\leqslant2r|\theta|+m_\theta(2r).
  \end{multline*}
\indent Let us proceed to show that, for every fixed $\omega\in\Omega_\theta$, the convergence \eqref{claim distilled} holds. We first take note of the following scaling relations:
\[u^\epsilon_\theta(t,x,\omega)=t(\epsilon/t)u_\theta(t/\epsilon,x/\epsilon,\omega)=tu^{\epsilon/t}(1,x/t,\omega)\quad\text{for all $t>0$ and $x\in\R^d$}.\]
Fix $T>0$. Then for $r\in(0,T)$ we obtain
\begin{align*}
  \sup_{r\leqslant t\leqslant T}\sup_{y\in B_R}&|u^\epsilon_\theta(t,y,\omega)-\dprod{\theta}{y}+t\overline{H}(\theta)|
\\&=\sup_{r\leqslant t\leqslant T}\sup_{y\in B_R}|tu^{\epsilon/t}_\theta(1,y/t,\omega)-t\dprod{\theta}{y/t}+t\overline{H}(\theta)|
\\&\leqslant T\sup_{\epsilon/T\leqslant \eta\leqslant \epsilon/r}\sup_{z\in B_{R/r}}|u^\eta_\theta(1,z,\omega)-\dprod{\theta}{z}+\overline{H}(\theta)|.
\end{align*}
By \eqref{t1}, the right-hand side goes to $0$ as $\epsilon\to 0^+$.
Finally, we use the uniform in $\epsilon$ (and $\omega$)
continuity of $u^\epsilon(t,0,\omega)$ at $t=0$ implied by the condition ${\cal K}_*$
and get that for all $r\in(0,T)$
\begin{align*}
  &\sup_{0\leqslant t\leqslant r}\sup_{y\in B_R}
  |u_\theta^\epsilon(t,y,\omega)-\dprod{\theta}{y}+
  t\overline{H}(\theta)|\\&\leqslant \sup_{0\leqslant t\le
                            r}\sup_{y\in B_R}|\epsilon u_\theta(t/\epsilon,y/\epsilon,\omega)-\epsilon u_\theta(0,y/\epsilon,\omega)|+r|\overline{H}(\theta)|\leqslant \epsilon a+rM_a+r|\overline{H}(\theta)|.
\end{align*}
The last expression goes to zero when we let $\epsilon\to 0^+$ and then
$r\to 0^+$. This proves \eqref{claim distilled} for all
$\omega\in\Omega_\theta$.  The remainder of the statement with
$\hat{\Omega}=\cap_{\theta\in\Q^d}\Omega_\theta$ follows from \eqref{claim distilled}, (L$'$), and the bound (3.6) in Proposition~\ref{app prop
  useful}.
\end{proof}

\subsection{Homogenization for non-convex Hamiltonians.}\label{sez non convex homogenization}
The aim of the present section is to establish a homogenization result
in the stationary ergodic setting for equations of the form 
(HJ$_\epsilon$) in one
space dimension, where the stationary random field
$H:\Omega\to \D{C}(\R\times\R)$ takes values in a special class of
non-convex Hamiltonians, see Theorem \ref{teo hom final} for
details. The proof of this result is derived from a more general
principle that we shall describe and prove first.

Let $H_+,H_-:\Omega\to\D{C}\left(\R\times\R\right)$ be
stationary random fields such that $H_\pm(\cdot,\cdot,\omega)\in\Ham$ for every $\omega$ (see Definition~\ref{Fam}), 
with $\gamma>1,\, \alpha_0,\,\beta_0>0$ independent of $\omega$. In
addition, we assume that 
\begin{equation}
  \label{hpm}
  H_\pm(x,0,\omega)= h_0\ \ \text{and}\ \  H_+(x,p,\omega)p\leqslant H_-(x,p,\omega)p\quad\text{for all $(x,p,\omega)\in\R^2\times\Omega$,}
\end{equation}
for some constant $h_0\in\R$.    
Let
\begin{equation}
  \label{Ham}
 H(x,p,\omega)=\min\{H_+(x,p,\omega),\ H_-(x,p,\omega)\}\quad\hbox{for all $(x,p,\omega)\in\R^2\times\Omega$}.
\end{equation}
Then $H(\cdot,0,\cdot)\equiv h_0$ on $\R\times\Omega$ 
and, in view of \eqref{hpm}, we also have
\[
 H(x,p,\omega)=H_+(x,p,\omega)\quad\hbox{if $p\geqslant 0$}
 \qquad\hbox{and}\qquad
 H(x,p,\omega)=H_-(x,p,\omega)\quad\hbox{if $p\leqslant 0$.}
\]
Note that $H$ is stationary and, for every fixed $\omega$, belongs to the same class
$\Ham$ as $H_\pm$.  We let $A(x,\omega)$ be a stationary process which satisfies
(A1)-(A2) with $\Lambda_A$ independent of $\omega$. 
We consider the family $u^\epsilon(\cdot,\cdot,\omega)$,
$\epsilon\in(0,1]$, of solutions to the equation
\begin{equation}\label{eq parabolic new}\tag{HJ$_\epsilon^\omega$}
u^\eps_t-\eps A\left(\frac x\eps\right) u^\eps_{x\,x}+H\left(\frac{x}{\eps}, u^\eps_x,\omega\right)=0\quad\hbox{in $(0,+\infty)\times\R$},
\end{equation}
subject to the initial condition $u^\epsilon(0,\cdot,\omega)=g\in \D{UC}(\R)$. These Cauchy problems are
$\K_*$-well-posed thanks to Theorem~\ref{teo main HJ inf H}. 

We aim to show that equation \eqref{intro eq parabolic} homogenizes
whenever this holds true with $H_{\pm}$ in place of
$H$. More precisely, we shall prove the following
  result.
\begin{teorema}\label{teo main hom}
  Let $H$ be given by \eqref{Ham}, where
  $H_\pm:\Omega\to\D{C}\left(\R\times\R\right)$ are
  stationary random fields satisfying \eqref{hpm} and such that
  $H_\pm(\cdot,\cdot,\omega)\in\Ham$ for every $\omega$, with
  constants $\gamma>1,\, \alpha_0,\,\beta_0>0$ independent of
  $\omega$.  Let us furthermore assume that
  there exist sets $\Omega_\pm$ of full measure in $\Omega$ such that
  \eqref{intro eq parabolic} with $H_\pm$ in place of $H$ homogenizes
  for all linear initial data $g(x)=\theta x$, $\theta\in\R$ and for
  every $\omega\in\Omega_\pm$, respectively.  Let us denote by
  $\overline H_\pm$ the associated (continuous and coercive) effective
  Hamiltonians.  Then there exists a continuous and coercive
  Hamiltonian $\overline H:\R\to\R$ such that, for every
  $\omega\in\Omega_-\cap\Omega_+$ and every initial datum
  $g\in\D{UC}(\R)$, the unique solutions
  $u^\eps(\cdot,\cdot,\omega)\in\D{UC}(\ccyl)$ of \eqref{intro eq
    parabolic} with the initial condition $u^\eps(0,\cdot,\omega)=g$
  converge, locally uniformly in $[0,+\infty)\times\R$ as
  $\eps\to 0^+$, to the unique solution $\overline u\in\D{UC}(\ccyl)$
  of
\begin{align*}
\begin{cases}
\overline u_t+\overline H(\overline u_x)=0 &\hbox{in $[0,+\infty)\times\R$}\\
\overline u(0,\cdot,\omega)=g & \hbox{in $\R$.}
\end{cases}
\end{align*}
Moreover,    
\[
\overline{H}(\theta)=\min\{\overline{H}_-(\theta),\overline{H}_+(\theta)\}\qquad\hbox{for every $\theta\in\R$}.
\] 
\end{teorema}

Before dealing with the proof of Theorem \ref{teo main hom}, we derive some simple consequences.

{\begin{cor}\label{cor Hpmconvex} 
Let $H$ be as in the statement of Theorem \ref{teo main hom} and let us additionally assume that 
$H_\pm$ are convex in $p$, or level-set convex in $p$ if
    $A\equiv 0$. Then there exists a set $\hat\Omega$ of full measure in $\Omega$ such that, for every $\omega\in\hat\Omega$, equation \eqref{intro eq parabolic} homogenizes for all $g\in\D{UC}(\R)$.
\end{cor}
}

\begin{proof}
  The assertion immediately follows from Theorem~\ref{teo main hom}
  and homogenization results for viscous and non-viscous
  Hamilton-Jacobi equations with stationary ergodic convex (or level
  set convex if $A\equiv 0$) Hamiltonians \cite{AT15, DS09, KRV,
    LS_viscous,RT,Sou99}.
\end{proof}

\begin{oss}\label{oss non convex effective H}
  Note that the effective Hamiltonian $\overline H$ associated to
 \eqref{intro eq parabolic} is, in general, neither convex nor even level-set convex. Indeed,
  let $b(\cdot,\omega)\in \D{Lip}(\R)$ with a Lipschitz constant
  independent of $\omega$, $a\leqslant b(x,\omega)\leqslant 1/a$ in
  $\R\times\Omega$ for some constant $a\in(0,1)$. Let $H_\pm(x,p,\omega)=|p|^2/2\mp b(x,\omega)p$ and define
  \[
  H(x,p,\omega)=\min\{H_+(x,p,\omega),H_-(x,p,\omega)\}=\frac{|p|^2}{2}-b(x,\omega)|p|\quad\text{in
    $\R^2\times\Omega$.}
  \]
  Then $H,H_\pm\in{\cal H}(2,\alpha_0,\beta_0)$ for some
  $\alpha_0,\beta_0>0$. Set $\hat H(p):=|p|^2/2-a|p|$. Let us denote
  by $u^\eps_\theta$ and $\hat u^\eps_\theta$ the solutions to
 \eqref{intro eq parabolic} with $A=1$ and Hamiltonians $H$ and $\hat H$
  respectively, and the initial condition $\theta x$. Since
  $H\leqslant \hat H$, by the Comparison Principle stated in
  Proposition~\ref{prop comp in Ham} we infer that
  $\hat u^\eps_\theta\leqslant u^\eps_\theta$. Therefore,
\[
 \overline H(\theta)=\lim_{\eps\to 0}-u^\eps(1,0,\omega)\leqslant \lim_{\eps\to 0}-\hat u^\eps(1,0)= \hat H(\theta)\qquad\hbox{a.s. in $\omega$,}
\]
in particular $\overline H(\pm a)\leqslant \hat H(\pm a)<0$. Since $\overline H(0)=0$ and $\overline H$ is coercive, the assertion follows.\medskip 
\end{oss}

We now proceed to prove Theorem \ref{teo main hom}.  The idea we are
going to exploit is that, if the initial datum $g$ is monotone, the
solution of the corresponding Cauchy problem associated to
\eqref{intro eq parabolic} enjoys the same kind of monotonicity as $g$
with respect to the $x$ variable.  Due to the scaling
  relation 
$u^\eps_{\theta}(t,x)=\eps u_\theta(t/\eps,x/\eps,\omega)$, it
suffices to consider the case $\eps=1$.

\begin{prop}\label{prop linear}
  Let $A$ satisfy (A1)-(A2) and $H\in\Ham$ {be such that
  $H(\cdot,0)\equiv h_0$ on $\R$, for some constant $h_0\in\R$}.  Denote by $w$ the unique
  continuous solution in $\K_*$ of the Cauchy problem
\begin{equation}\label{eq linear}
\
\begin{cases}
w_t- A(x)w_{xx}+H(x,w_x)=0&\hbox{in $(0,+\infty)\times\R$},  \\
w(0,x)=g(x)& \text{in $\R$},
\end{cases}
\end{equation}
where $g\in{\Lip^2}(\R^d)$.
Then
\begin{itemize}
\item[\em (i)] if $g'\geqslant 0$, then $w(t,\cdot)$ is nondecreasing, for every fixed $t\geqslant 0$;\smallskip
\item[\em (ii)] if $g'\leqslant 0$, then $w(t,\cdot)$ is nonincreasing, for every fixed $t\geqslant 0$.
\end{itemize}
\end{prop}

\begin{proof}
  The existence and uniqueness of continuous $w$ in $\K_*$ follow from
  Theorem~\ref{teo main HJ inf H}. By the same theorem, $w$ is
  Lipschitz continuous in $\ccyl$. 

  Let us now show monotonicity. Up to replacing $H$ with $H-h_0$ and
  $w$ with $w+h_0t$, we can  assume without 
    loss of generality that $H(\cdot,0)\equiv 0$ on $\R$. Suppose, for definiteness, that 
    $g'\geqslant 0$. For technical reasons we need to perform a
  regularization of $g$, $A$ and of the nonlinearity $H$. To this aim,
  take a standard sequence of mollifiers $(\rho_n)_n$ on $\R$ and set
  $g_n(x):=g\ast \rho_n(x)$ $A_n(x)=A\ast \rho_n(x)+1/n$ and
  $H_n(x,p):=\hat H_n(x,p)-\hat H_n(x,0)$ with
\[
 \hat H_n(x,p):=\int_{\R\times\R} \rho_n(y)\rho_n(q)H(x-y,p-q)\,\dd y\,\dd q,
 \quad\hbox{$(x,p)\in\R\times\R.$}
\]
An easy check shows that $A_n$ satisfies (A1)-(A2) with
$\Lambda_{A_n}=\Lambda_A+1$ and that $H_n$ belongs to the
class $\Ham$  {with} constants $\alpha_0,\beta_0$ independent of $n$, possibly different from the ones assumed for 
$H$ at the beginning. Thus, the conditions of Theorem~\ref{teo main HJ
  inf H} are satisfied. Moreover, for every $R>0$, there exists a
constant $C(R,n)>0$ such that
\begin{equation}\label{auxiliary estimates}
  |\partial_x A_n|,\ |\partial_p H_n|,\ |\partial^2_p H_n|,\ |\partial^2_{x\,p} H_n|\leqslant C(R,n)\qquad\hbox{for all $x\in\R$ and $p\in B_R$.}
\end{equation}
Let us denote by $w_n$ the unique continuous solution of class $\K_*$
of the problem \eqref{eq linear} with $H_n$ in place of $H$.
Since $\sup_n \|Dg_n\|_\infty+\|D^2g_n\|_\infty<+\infty$, from 
Theorem~\ref{teo main HJ inf H} and Remark~\ref{oss Lip dependence} 
we infer that the solutions $w_n$ are equi-Lipschitz in $[0,+\infty)\times\R$.

By standard regularity results for parabolic equations,  we know that $w_n$ is also a smooth, classical solution of the
problem \eqref{eq linear} with $H_n$ in place of $H$.  
Moreover, since
\begin{equation}\label{en}
A_n(x)\left(w_n\right)_{xx}= \left(w_n\right)_t+H_n(x,\left(w_n\right)_x)\qquad \hbox{for all $(t,x)\in(0,+\infty)\times\R$,}
\end{equation}
$A_n(x)\geqslant 1/n$, and $w_n\in\Lip([0,+\infty)\times\R)$, we
infer that $\left(w_n\right)_{xx}$ is bounded on $[0,+\infty)\times\R$
(with a bound depending on $n$). By differentiating equation \eqref{en}
with respect to $x$, we get that the function
$v_n:=\left(w_n\right)_x$ solves the following Cauchy problem:
\begin{eqnarray}\label{auxiliary cauchy pb}
\begin{cases}
\left(v_n\right)_{t}- A_n(x) \left(v_n\right)_{xx}+I_n(x,v_n,\left(v_n\right)_x)=0 & \hbox{in $(0,+\infty)\times\R$}\medskip\\
v_n(0,x)=g_n'(x)\geqslant 0& \hbox{in $\R$,}
\end{cases}
\end{eqnarray}
with $I_n(x,p,\xi):=\partial_x H_n(x,p)+(\partial_p H_n(x,p)-\partial_xA_n(x))\xi$ for $(x,p,\xi)\in\R\times\R\times\R$. Now notice that $\partial_x H_n(x,0)=0$ for every $x\in\R$ since $H_n(\cdot,0)\equiv 0$ on $\R$, in particular $I_n(x,0,0)=0$. So we have  
\begin{eqnarray*}
I_n(x,v_n,\left(v_n\right)_x)=\int_0^1\frac{\dd}{\dd s} \big(I_n(x,sv_n,s\left(v_n\right)_x)\big)\,\dd s=b_n(t,x) \left(v_n\right)_x+c_n(t,x)v_n 
\end{eqnarray*}
with
\begin{align*}
b_n(t,x):=\int_0^1 \partial_\xi I_n\big(x,sv_n(t,x),s\left(v_n\right)_x(t,x)\big)\,\dd s\\
c_n(t,x):=\int_0^1 \partial_p I_n\big(x,sv_n(t,x),s\left(v_n\right)_x(t,x)\big)\,\dd s.
\end{align*}
Therefore $v_n$ is a bounded solution of the following linear {homogeneous}
parabolic equation
\[
\left(v_n\right)_{t}- A_n(x)\left(v_n\right)_{xx}+b_n(t,x) \left(v_n\right)_x+c_n(t,x)v_n=0 \qquad \hbox{in $(0,+\infty)\times\R$} 
\]
with coefficients $b_n(t,x)$ and $c_n(t,x)$ that are continuous and
bounded in $(0,+\infty)\times\R$, in view of \eqref{auxiliary
  estimates} and of the fact that $v_n$ and $\left(v_n\right)_x$ are
bounded on $[0,+\infty)\times\R$.  By the classical maximum principle
(see, for instance, \cite[Ch.\,2, Sec.\,4, Th.\,9]{Friedman}) we
conclude that $\left(w_n\right)_x(t,x)=v_n(t,x)\geqslant 0$ in
$[0,+\infty)\times\R$, thus, proving the asserted monotonicity of
$w_n$ in $x$. Now we pass to the limit in $n$: since $g_n\ucv g$ and
$A_n\ucv A$ in $\R$, $H_n\ucv H$ in $\R\times\R$, and the functions
$w_n$ are equi-Lipschitz in $[0,+\infty)\times\R$, we infer that
$w_n\ucv w$ in $[0,+\infty)\times\R$. The desired
monotonicity of $w(t,\cdot)$ {follows}.  The proof of  (ii) is analogous.
\end{proof}  

As an easy consequence, we derive the following crucial result. In
particular, it implies that condition (L$'$) holds with $\kappa_r$
independent of $\omega$.

\begin{prop}\label{prop Lip estimates}
  For every $\theta\in\R$, let us denote by $u_\theta(t,x,\omega)$ the
  unique uniformly continuous solution of (HJ$_1^\omega$)
  satisfying $u_\theta(0,x,\omega)=\theta x$ for $x\in\R$ and
  $\omega\in\Omega$. The following holds:
\begin{itemize}
\item[\em (i)] if $\theta\geqslant 0$, then $u_\theta$ is also a solution of 
\begin{equation*}
u_t- A(x,\omega)u_{xx}+H_+(x,u_x,\omega)=0\qquad \hbox{in $(0,+\infty)\times\R$};
\end{equation*}
\item[\em (ii)] if $\theta\leqslant 0$, then $u_\theta$ is also a solution of 
\begin{equation*}
u_t- A(x,\omega)u_{xx}+H_-(x,u_x,\omega)=0\qquad \hbox{in $(0,+\infty)\times\R$}.
\end{equation*}
\end{itemize}
Moreover, for every $\omega$, the function $u_\theta(\cdot,\cdot,\omega)$ is Lipschitz continuous on $[0,+\infty)\times\R$, with a Lipschitz constant independent of $\omega$ and locally bounded with respect to $\theta\in\R^d$. 
\end{prop}

\begin{proof}
  Let $\omega$ be fixed. Items (i) and (ii) follow by applying Proposition \ref{prop linear} with $g(x):=\theta x$ 
  and from the fact that $H(x,p,\omega)=H_+(x,p,\omega)$ for $p\geqslant 0$, 
  $H(x,p,\omega)=H_-(x,p,\omega)$ for $p\leqslant 0$.  The asserted Lipschitz estimates for
  $u_\theta(\cdot,\cdot,\omega)$ follow from Theorem~\ref{teo main HJ
    inf H} and Remark \ref{oss Lip dependence}.
\end{proof}

\begin{oss}\label{oss H_pm}
The Lipschitz properties of $u_\theta(\cdot,\cdot,\omega)$ stated in the above proposition are 
independent of the fact that $H$ is of the form \eqref{Ham}. Indeed, it suffices that 
$H(\cdot, \cdot, \omega)$ belongs to $\Ham$, with constants $\gamma>1,\, \alpha_0,\,\beta_0>0$ independent of $\omega$. 
In particular, if the equation \eqref{intro eq parabolic} associated to such a $H$ homogenizes, the effective Hamiltonian 
$\overline H$ is continuous and coercive, according to Proposition \ref{app prop useful}.
\end{oss} 

We are now in a position to prove our basic
homogenization result.\smallskip

\noindent{\em Proof of Theorem \ref{teo main hom}.} 
  Let us fix $\theta\in\R$ and denote by
  $u^\eps_\theta(\cdot,\cdot,\omega)$ the unique uniformly continuous
  solution of\eqref{intro eq parabolic} satisfying
  $u^\eps_\theta(0,x,\omega)=\theta x$ in $\R$. Notice that
  $u^\eps_\theta$ is identically 0 when $\theta=0$. Let us then assume
  $\theta\not=0$. Since
  $u^\eps_{\theta}(t,x)=\eps u_\theta(t/\eps,x/\eps,\omega)$, it is
  easily seen, in view of Proposition \ref{prop Lip estimates}, that
  $u^\eps_\theta(\cdot,\cdot,\omega)$ solves, for every fixed
  $\omega$, the following equation
\begin{equation}\label{eq convex parabolic}
u^\eps_t- \eps A\left(\frac{x}{\epsilon},\omega\right)u^\eps_{xx}+H_\pm\left(\frac{x}{\eps}, u^\eps_x,\omega\right)=0\quad\hbox{in $(0,+\infty)\times\R$},
\end{equation}
according to the sign of $\theta$. By our assumptions, for each
$\omega\in\Omega_+\cap\Omega_-$ both equations in \eqref{eq convex
  parabolic} homogenize for all linear initial data $g(x)=\theta x$
with coercive and continuous $\overline{H}_\pm$, see Remark \ref{oss
  H_pm}.  Let us set \ $\overline H(p):=H_+(p)$ if $p\geqslant 0$ and
\ $\overline H(p):=H_-(p)$ if $p\leqslant 0$. The Hamiltonian
$\overline H$ is coercive and continuous, since
$\overline H_+(0)=\overline H_-(0)=0$.  Moreover, the functions
$u^\eps(\cdot,\cdot,\omega)$ satisfy condition (L$'$), for a constant
$\kappa_r$ independent of $\omega$, in view of Proposition~\ref{prop
  Lip estimates}. By Theorem~\ref{app teo hom} and
Proposition~\ref{app prop useful} the asserted homogenization result
follows.
 
To prove the formula for $\overline H$, we first remark that, since
$\overline{H}(\theta)=\overline{H}_+(\theta)$ for $\theta\geqslant 0$
and $\overline{H}(\theta)=\overline{H}_-(\theta)$ for
$\theta\leqslant 0$, we have
$\overline{H}(\theta)\geqslant\min\{\overline{H}_-(\theta),\overline{H}_+(\theta)\}$.
On the other hand, by comparison, we infer
$u^\epsilon_\theta\geqslant u^\epsilon_{\theta\pm}$, where
$u^\epsilon_{\theta\pm}$ are respectively solutions of \eqref{eq
  convex parabolic} with the initial datum $\theta x$. From the fact
that
\[
    \overline H(\theta)=\lim_{\eps\to 0^+} -u^\eps_\theta(1,0,\omega),
    \qquad
    {\overline H}_\pm(\theta)=\lim_{\eps\to 0^+} -u^\eps_{\theta\pm}(1,0,\omega)\qquad\hbox{a.s. in $\Omega$,}
\]
we also get the opposite inequality.\qed\\

Next, we extend the homogenization result stated in Corollary \ref{cor Hpmconvex} to the case of non-convex
  Hamiltonians with multiple {\em pinning points}.

\begin{definition}\label{def pinned}
  Let $H:\Omega\to \D{C}(\R^d\times\R^d)$ be a measurable random field. We shall say that $H(x,p,\omega)$ is {\em
  pinned} at $p_0$ if there is a constant $h_0\in\R$ such that
$H(\cdot,p_0,\cdot)\equiv h_0$ \ on $\R\times\Omega$. 
\end{definition}

The following simple observation will be very useful.  

\begin{oss}\label{shifts}
  If $H$ is pinned at $p_0\not=0$, then
  $\tilde H(x,p,\omega):=H(x,p+p_0,\omega)$ is pinned at
  $0$. Moreover, the solutions $u$ and $\tilde u$ of (HJ$_1^\omega$)
  with Hamiltonians $H$ and $\tilde H$, respectively, are related as
  follows:
\[u(t,x,\omega)=\tilde u(t,x,\omega)+p_0x.
\]
This implies that \eqref{intro eq parabolic} with Hamiltonian $H$ homogenizes if
and only if \eqref{intro eq parabolic} with Hamiltonian $\tilde H$ homogenizes.
\end{oss}
We are now ready to
  prove our final homogenization result.
\begin{teorema}\label{teo hom final}
\begin{sloppypar}
Let $H:\Omega\to\D{C}\left(\R\times\R\right)$ be a stationary random
  field satisfying $H(\cdot,\cdot,\omega)\in\Ham$ for every $\omega$,
  with constants $\gamma>1,\, \alpha_0,\,\beta_0>0$ independent of
  $\omega$.  Let us furthermore assume that
\end{sloppypar}
\begin{itemize}
 \item[(i)] $H$ is pinned at $p_1<p_2<\dots<p_n$;\smallskip 
 \item[(ii)] $H(x,\cdot,\omega)$ is convex (or level-set convex if
$A\equiv 0$) on each of the intervals $(-\infty,p_1),$\ $(p_1,p_2)$,\
$\dots,\ (p_{n},+\infty)$, for every $(x,\omega)\in\R\times\Omega$.\smallskip 
\end{itemize}
Then there exists a set $\hat\Omega$ of full measure in $\Omega$ 
such that, for every $\omega\in\hat\Omega$, equation \eqref{intro eq parabolic} homogenizes   
for all $g\in\D{UC}(\R)$.
\end{teorema}

\begin{proof}
Let us assume, to fix ideas, that $H(x,\cdot,\omega)$ is convex on each of the intervals $(-\infty,p_1),\ (p_1,p_2)$,\
$\dots,\ (p_{n},+\infty)$, for every $(x,\omega)\in\R\times\Omega$.
It follows from the hypotheses that $H$ can be written as 
\begin{equation*}\label{H multimin}
  H(x,p,\omega)=\min\{H_i(x,p,\omega)\,:\,\ 1\leqslant i\leqslant n+1\}=
\begin{cases}
  H_1(x,p,\omega)&\text{if }p\leqslant p_1;
  \\H_2(x,p,\omega)&\text{if }p_1\leqslant p\leqslant p_2;
  \\
  \dots & \dots\\
  H_{n+1}(x,p,\omega)&\text{if }p\geqslant p_{n},
\end{cases}
\end{equation*}
where  
$H_1,\dots,H_{n+1}:\Omega\to\D{C}\left(\R\times\R\right)$ 
are stationary random fields such that, for every fixed $\omega$, each $H_i(\cdot,\cdot,\omega)$ is a convex 
Hamiltonian belonging to $\Hamtilde$, with same  $\gamma$ and possibly different constants $\tilde\alpha_0,\tilde\beta_0>0$, 
independent of $\omega$ and $i\in \{1,\dots,n+1\}$. 

The proof is by induction on $n$. The case $n=1$ follows from
Corollary~\ref{cor Hpmconvex} and Remark~\ref{shifts}. Let us now
assume that the assertion holds
for $n-1$ and prove it for $n$. To this aim, notice that
$H$ can be written as
\begin{equation*}
  H(x,p,\omega)=\min\{H_-(x,p,\omega),\, H_+(x,p,\omega)\}=
\begin{cases}
  H_-(x,p,\omega)&\text{if }p\leqslant p_{n};\\
  H_+(x,p,\omega)&\text{if }p\geqslant p_{n},
\end{cases}
\end{equation*}
with $H_+:=H_{n+1}$ and $H_-$ defined as 
\begin{equation*}
  H_-(x,p,\omega):=
\begin{cases}
  H_1(x,p,\omega)&\text{if }p\leqslant p_1;
  \\H_2(x,p,\omega)&\text{if }p_1\leqslant p\leqslant p_2;
  \\
  \dots & \dots\\
  H_{n}(x,p,\omega)&\text{if }p\geqslant p_{n-1}.
\end{cases}
\end{equation*}
By the inductive hypothesis, the assertion holds both for $H_-$ and $H_+$. Another
application of Remark~\ref{shifts} together with Theorem \ref{teo hom final} 
yields that the assertion holds for $H$ as well. 
The case when $A\equiv 0$ and $H$ is level-set convex in $p$ on each of the intervals $(-\infty,p_1),\ (p_1,p_2)$,\
$\dots,\ (p_{n},+\infty)$ can be handled similarly. 
\end{proof}

\begin{oss}
  Let $H,\,H_1,\dots,H_{n+1}$ be as above.  By the same
  reasoning as in the proof of Theorem \ref{teo main hom} and {by} induction, it can
be easily shown that the effective Hamiltonian $\overline H$
associated to $H$ satisfies
\[
 \overline H(\theta)=\min\{\overline H_1(\theta),\dots,\overline H_{n+1}(\theta)\}
 \qquad
 \hbox{for every $\theta\in\R$,}
\]
where $\overline H_i$ is the effective Hamiltonian associated to $H_i$, for each $i\in\{1,\dots,n+1\}$. 
\end{oss}

\bibliography{viscousHJ}
\bibliographystyle{siam}

\end{document}